\documentclass[preprint,12pt]{elsarticle}
\usepackage{amsfonts}
\pdfoutput=1
\usepackage[utf8]{inputenc}
\usepackage{amssymb}
\usepackage{amsthm}
\usepackage{mathrsfs}
\usepackage{amsmath}
\usepackage{graphicx}
\usepackage{color}

\newtheorem{theorem}{Theorem}[section]
\newtheorem{lemma}[theorem]{Lemma}

\theoremstyle{definition}

\newtheorem{example}[theorem]{Example}

\theoremstyle{remark}
\newtheorem{remark}[theorem]{Remark}
\numberwithin{equation}{section}

\begin{document}

\begin{frontmatter}




\title{The number of solutions of diagonal cubic equations over finite fields}

\author[]{Wenxu Ge\corref{cor1}}
\ead{gewenxu@ncwu.edu.cn}
\cortext[cor1]{Corresponding author}

\address{School of Mathematics and Statistics, North China University of Water Resources and Electric Power,
              Zhengzhou 450046, P.R.China}

%\fntext[label1]{Corresponding author}
\author[]{Weiping Li}
\ead{wpliyh@163.com}
\address{School of Mathematics and Information Sciences, Henan University of Economics and Law,\\
Zhengzhou, 450046, P.R.China}

\author[]{Tianze Wang}
\ead{wtz@ncwu.edu.cn}
\address{School of Mathematics and Statistics, North China University of Water Resources and Electric Power, \\
Zhengzhou, 450046, P.R.China}

\begin{abstract}
Let $\mathbb{F}_q$ be a finite field of $q=p^k$ elements. For any $z\in \mathbb{F}_q$, let $A_n(z)$ and $B_n(z)$ denote the number of solutions of the equations
$x_1^3+x_2^3+\cdots+x_n^3=z$ and $x_1^3+x_2^3+\cdots+x_n^3+zx_{n+1}^3=0$ respectively. Recently, using the generator of $\mathbb{F}^{\ast}_q$, Hong and Zhu gave the generating functions $\sum_{n=1}^{\infty}A_n(z)x^n$ and $\sum_{n=1}^{\infty}B_n(z)x^n$. In this paper, we give the generating functions $\sum_{n=1}^{\infty}A_n(z)x^n$ and $\sum_{n=1}^{\infty}B_n(z)x^n$ immediately by the coefficient $z$. Moreover, we gave the formulas of the number of solutions of equation $a_1x_1^3+a_2x_2^3+a_3x_3^3=0$ and our formulas are immediately determined by the coefficients $a_1,a_2$ and $a_3$. These extend and improve earlier results.
\end{abstract}


\end{frontmatter}


\section{Introduction}

Let $\mathbb{F}_q$ be a finite field of $q=p^k$ elements. Let $\mathbb{F}^{\ast}_q$ be the multiplicative group of $\mathbb{F}_q$, i.s.
$\mathbb{F}^{\ast}_q=\mathbb{F}_q\setminus\{0\}$. Counting the number of solutions $(x_1,x_2,\cdots,x_n)\in \mathbb{F}^n_q$ of the general diagonal equation
$$a_1x_1^{d_1}+a_2x_2^{d_2}+\cdots+a_nx_n^{d_n}=b$$
over $\mathbb{F}_q$ is an important and fundamental problem in number theory and finite field. The special case where all the $d_i$ are equal has extensively been studied by many authors (see, for example, \cite{J,Mor,M1,Wan,Weil,W1}).

For any $z\in \mathbb{F}_q$, one lets $A_n(z)$ denote the number of solutions of the following diagonal equation
$$
x_1^3+x_2^3+\cdots+x_n^3=z
$$
over $\mathbb{F}_q$. When $q=p\equiv 1(\bmod 3)$, Chowla, Cowles and Cowles \cite{CCC1} gave the generating function $\sum_{n=0}^{\infty} A_n(0)x^n$.
Myerson \cite{M1} extended the Chowla, Cowles and Cowles's result to finite field $\mathbb{F}_q$. He proved the following result.

\begin{theorem}[\cite{M1}]
Let $\mathbb{F}_q$ be a finite field of $q=p^k$ elements with $q\equiv 1(\bmod 3)$. Then
\begin{align*}
\sum_{n=1}^{\infty}A_n(0)x^n=\frac{x}{1-qx}+\frac{(q-1)(2+cx)x^2}{1-3qx^2-qcx^3},
\end{align*}
where $c$ is uniquely determined by
\begin{align}\label{c}
4q=c^2+27d^2, c\equiv 1 (\bmod 3)\  \mathrm{and}\  \mathrm{if}\  p\equiv 1 (\bmod 3),\  \mathrm{then}\  (c,p)=1.
\end{align}
\end{theorem}

Recently, Hong and Zhu \cite{HZ} consider $A_n(z)$ in finite field $\mathbb{F}_q$, they proved the following result.

\begin{theorem}[\cite{HZ}]
Let $z\in \mathbb{F}^{\ast}_q=\langle g\rangle$ and $q=p^k\equiv 1(\bmod 3)$ with $k$ being a positive integer. Then
$$
\sum_{s=1}^{\infty}A_s(z)x^s=\frac{x}{1-qx}+\frac{2x+(c-2)x^2-cx^3}{1-3qx^2-qcx^3}
$$
if $z$ is cubic, where $c$ is uniquely determined by (\ref{c}), and
$$
\sum_{s=1}^{\infty}A_s(z)x^s=\frac{x}{1-qx}-\frac{x+\frac{1}{2}(4+c+9d\delta_z(d))x^2+cx^3}{1-3qx^2-qcx^3}
$$
if $z$ is non-cubic, where $c$ and $d$ are uniquely determined by (\ref{c}) with $d>0$ and
\begin{align}\label{q}
\delta_z(q)=\left\{
                \begin{array}{ll}
                (-1)^{\langle ind_g(d)\rangle_3}\cdot sgn\left(\mathrm{Im}(r_1+3\sqrt{3}r_2\mathrm{i})^k\right), & \hbox{if $k\equiv 1(\bmod 2)$;} \\
                 0, & \hbox{if $k\equiv 0(\bmod 2)$.}
                \end{array}
              \right.
\end{align}
where $r_1$ and $r_2$ are uniquely determined by
$$4p=r_1^2+27r_2^2,\ \ r_1\equiv 1(\mathrm{mod}3),\ \ 9r_2\equiv (2\mathrm{N}_{\mathbb{F}_q/\mathbb{F}_p}(g)^{\frac{p-1}{3}}+1)r_1 (\bmod p).$$
\end{theorem}

Suppose that $z\in \mathbb{F}^{\ast}_q$ be non-cubic.
Let $B_n(z)$ be the number of solutions of diagonal cubic equation
\begin{align*}
x_1^3+x_2^3+\cdots+x_n^3+zx^3_{n+1}=0
\end{align*}
over $\mathbb{F}_q$.
In \cite{HZ}, Hong and Zhu also consider $B_n(z)$. They showed the following result.
\begin{theorem}[\cite{HZ}]
Let $z\in \mathbb{F}^{\ast}_q$ be non-cubic and $q=p^k\equiv 1(\bmod 3)$ with $k$ being a positive integer. Then
$$
\sum_{s=1}^{\infty}B_s(z)x^s=\frac{qx}{1-qx}-\frac{(q-1)x+\frac{1}{2}(q-1)(c-9d)x^2}{1-3qx^2-qcx^3},
$$
where $c$ and $d$ are uniquely determined by (\ref{c}) with $d>0$ and $\delta_z(q)$ is given as in (\ref{q}).
\end{theorem}

Indeed, The key of these problems is to determine the sign of $d$. In Hong and Zhu's results, they use the generator of group $\mathbb{F}^{\ast}_q$ to determine the sign of $d$. However, for a large prime $p$, it is not easy to find a generator of group $\mathbb{F}^{\ast}_q$. In this paper, by calculating the Jacobi sum of finite field, we
determine the sign of $d$ immediately by the coefficient $z$. We give the following two results.

\begin{theorem}
Let $\mathbb{F}_q$ be a finite field of $q=p^k$ elements with $q\equiv 1(\bmod 3)$. Then
$$
\sum_{n=1}^{\infty}A_n(z)x^n=\frac{x}{1-qx}+\frac{2x+(c-2)x^2-cx^3}{1-3qx^2-qcx^3}
$$
if $z$ is cubic, and
$$
\sum_{n=1}^{\infty}A_n(z)x^n=\frac{x}{1-qx}-\frac{x+\frac{1}{2}(4+c-9d)x^2+cx^3}{1-3qx^2-qcx^3}
$$
if $z$ is non-cubic, where $c$ and $d$ are uniquely determined by
\begin{align}\label{cd}
4q=c^2+27d^2, c\equiv 1 (\bmod 3),  (c,p)=1, 9d\equiv c(2z^{\frac{q-1}{3}}+1)(\bmod p).
\end{align}
\end{theorem}

\begin{theorem}
Let $\mathbb{F}_q$ be a finite field of $q=p^k$ elements with $q\equiv 1(\bmod 3)$ and $z\in \mathbb{F}^{\ast}_q$ be non-cubic. Then we have
\begin{align*}
\sum_{n=0}^{\infty}B_n(z)x^n=\frac{1}{1-qx}-\frac{(q-1)x+\frac{1}{2}(q-1)(c-9d)x^2}{1-3qx^2-qcx^3},
\end{align*}
where $c$ and $d$ are uniquely determined by (\ref{cd}).
\end{theorem}

\begin{remark}
When $q\equiv 2(\bmod 3)$, it is known that every element is a cube, so $N_n(z)=q^{n-1}$. If $q\equiv 1(\bmod 3)$ with
$p\equiv 2(\bmod 3)$, then Wolfmann \cite{W2} gave a formula for $N_n(z)$. By Theorem 16 of \cite{M2}, we have
\begin{align*}
c=\left\{
    \begin{array}{ll}
      -2p^{k/2}, & \hbox{if $k\equiv 0(\bmod 4)$;} \\
      2p^{k/2}, & \hbox{if $k\equiv 2(\bmod 4)$,}
    \end{array}
  \right.
\end{align*}
and $d=0$. Then for this case, Theorem 1.4 and 1.5 immediately follow from Theorem 1.2 and 1.3.
So in the rest of this paper, we focus on the case $q\equiv 1(\bmod 3)$ with $p\equiv 1(\bmod 3)$.

\end{remark}

For $a_1,a_2,a_3\in \mathbb{F}^{\ast}_q$,
let $M_k(a_1,a_2,a_3)$ be the number of solutions of
\begin{align*}
a_1x_1^3+a_2x_2^3+a_3x_3^3=0
\end{align*}
over $\mathbb{F}_q$ and let $N_k(a_1,a_2,a_3)$ be the number of solutions of
\begin{align*}
a_1x_1^3+a_2x_2^3=a_3
\end{align*}
over $\mathbb{F}_q$.
For the case $q=p\equiv 1(\bmod 3)$, Chowla, Cowles and Cowles \cite{CCC1} showed that $M_1(1,1,1)=p^2+c(p-1)$.
As pointed out in \cite{CCC2}, the following is essentially included in the derivation of the cubic equation of periods by Gauss \cite{Gau}: Let a prime
$p\equiv 1(\bmod 3)$ and $z$ be non-cubic in $\mathbb{F}_p$. Then one has
$$M_1(1,1,z)=p^2+\frac{1}{2}(p-1)(9d-c),$$
where $c$ and $d$ are uniquely determined by (\ref{c}) (except for the sign of $d$).

Chowla, Cowles and Cowles \cite{CCC2} determined the sign of $d$ for the case of 2 being non-cubic in $\mathbb{F}_p$.

\begin{theorem}[\cite{CCC2}]
Let a prime $p\equiv 1(\bmod 3)$. If 2 is non-cubic in $\mathbb{F}_p$, then for any non-cubic element $z$, one has
$$M_1(1,1,z)=p^2+\frac{1}{2}(p-1)(9d-c),$$
where $c$ and $d$ are uniquely determined by (\ref{c}) with
$$d\equiv c (\bmod 4)\ \  \mathrm{if}\ 4z\  \mathrm{is} \ \mathrm{cubic}$$
and
$$d\equiv -c (\bmod 4)\ \  \mathrm{if}\ 2z\  \mathrm{is} \ \mathrm{cubic}.$$
\end{theorem}

In \cite{HZ}, Hong and Zhu solved the Gauss sign problem. In fact, they gave the following result.

\begin{theorem}[\cite{HZ}]
Let $z\in \mathbb{F}^{\ast}_q=\langle g\rangle$ be non-cubic and $q=p^k\equiv 1(\bmod 3)$. Then
$$M_k(1,1,z)=q^2+\frac{1}{2}(q-1)(-c-9\delta_z(q)d),$$
where $c$ and $d$ are uniquely determined by (\ref{c}) with $d>0$ and $\delta_z(q)$ is given as in (\ref{q}).

\end{theorem}
In this paper, we consider $M_k(a_1,a_2,a_3)$, $N_k(a_1,a_2,a_3)$ and determine the sign of $d$ immediately by the coefficients $a_1,a_2$ and $a_3$. We have the following more general results.

\begin{theorem}
Let $\mathbb{F}_q$ be a finite field of $q=p^k$ elements with the prime $p\equiv 1(\bmod 3)$, and $a_1,a_2,a_3\in \mathbb{F}^{\ast}_q$.
Then
\begin{align*}
M_k(a_1,a_2,a_3)=\left\{
                   \begin{array}{ll}
                     q^2+c(q-1), & \hbox{if $a_1a_2a_3$ is cubic;} \\
                     q^2+\frac{1}{2}(q-1)(9d-c), & \hbox{if $a_1a_2a_3$ is non-cubic,}
                   \end{array}
                 \right.
\end{align*}
where $c$ and $d$ are uniquely determined by
\begin{align}\label{cd for a}
4q=c^2+27d^2, c\equiv 1(\bmod 3), (c,p)=1, 9d\equiv c(2(a_1a_2a_3)^{\frac{q-1}{3}}+1)(\bmod p).
\end{align}

\end{theorem}

\begin{theorem}
Let $\mathbb{F}_q$ be a finite field of $q=p^k$ elements with the prime $p\equiv 1(\bmod 3)$, and $a_1,a_2,a_3\in \mathbb{F}^{\ast}_q$.

(1) For the case of $a_1a_2a_3$ being cubic, we have
$$N(a_1,a_2,a_3)=\left\{
                   \begin{array}{ll}
                     q-2+c, & \hbox{if $a_1a_2^{-1}$ is cubic;} \\
                     q+1+c, & \hbox{otherwise.}
                   \end{array}
                 \right.
$$

(2) For the case of $a_1a_2a_3$ being non-cubic, we have
$$N(a_1,a_2,a_3)=\left\{
                   \begin{array}{ll}
                    q-2+\frac{1}{2}(9d-c), & \hbox{if $a_1a_2^{-1}$ is cubic;} \\
                    q+1+\frac{1}{2}(9d-c), & \hbox{otherwise,}
                   \end{array}
                 \right.
$$
where $c$ and $d$ are uniquely determined by (\ref{cd for a}).

\end{theorem}

\section{Auxiliary Lemmas}

\begin{lemma}[\cite{LN}]
Let $\mathbb{F}_q$ be a finite field. Let $\chi$ be a nontrivial multiplicative character of $\mathbb{F}_q$ and $\psi$ be a nontrivial additive character of $\mathbb{F}_q$. Then for any $a\in \mathbb{F}_q$, we have
$$
\sum_{x\in \mathbb{F}^{\ast}_q}\chi(x)=0,\ \
\sum_{x\in \mathbb{F}_q}\psi(ax)=\left\{
                                   \begin{array}{ll}
                                     q, & \hbox{if $a=0$;} \\
                                     0, & \hbox{if $a\neq 0$.}
                                   \end{array}
                                 \right.
$$
\end{lemma}

For any $a\in \mathbb{F}^{\ast}_q$, we defined the sums
\begin{align*}
S(a)=\sum_{x\in \mathbb{F}_q}\psi(ax^3)
\end{align*}
and
\begin{align*}
G(\chi,\psi)=\sum_{x \in \mathbb{F}^{\ast}_q}\chi(x)\psi(x),
\end{align*}
where $\chi$ is a multiplicative character of $\mathbb{F}_q$ and $\psi$ is an additive character of $\mathbb{F}_q$. Both $S(a)$ and $G(\chi)$ are called Gauss sums.

\begin{lemma}[\cite{LN}]
Let $\chi$ be a nontrivial multiplicative character and $\psi$ a nontrivial additive character of $\mathbb{F}_q$. Then $|G(\chi,\psi)|=\sqrt{q}$ and $G(\chi,\psi)G(\overline{\chi},\psi)=\chi(-1)q$.
\end{lemma}

Let $\mathbb{F}_q$ be the finite extension of $\mathbb{F}_p$ with $[\mathbb{F}_q:\mathbb{F}_p]=k$. Recall that the trace $\mathrm{Tr}_{\mathbb{F}_q/\mathbb{F}_p} (\alpha)$ and norm $\mathrm{N}_{\mathbb{F}_q/\mathbb{F}_p} (\alpha)$ of $\alpha\in \mathbb{F}_q$ over $\mathbb{F}_p$ are defined by
\begin{align*}
\mathrm{Tr}_{\mathbb{F}_q/\mathbb{F}_p} (\alpha)=\alpha+\alpha^p+\cdots+\alpha^{p^{k-1}}
\end{align*}
and
\begin{align*}
\mathrm{N}_{\mathbb{F}_q/\mathbb{F}_p} (\alpha)=\alpha\times\alpha^p\times\cdots\times\alpha^{p^{k-1}}=\alpha^{\frac{q-1}{p-1}}.
\end{align*}

\begin{lemma}[Hasse-Davenport Theorem \cite{LN}]
Let $\mathbb{F}_q$ be the finite extension of $\mathbb{F}_p$ with $[\mathbb{F}_q:\mathbb{F}_p]=k$. Let $\chi'$ be a multiplicative character and $\psi'$ an additive character of $\mathbb{F}_p$, not both of them trivial. Suppose that $\chi$ and $\psi$ are the lifts of $\chi'$ and $\psi'$ from $\mathbb{F}_p$ to $\mathbb{F}_q$, i.e. $\chi=\chi'\circ\mathrm{N}_{\mathbb{F}_q/\mathbb{F}_p}$ and $\psi=\psi'\circ\mathrm{Tr}_{\mathbb{F}_q/\mathbb{F}_p}$. Then
$$G(\chi,\psi)=(-1)^{k-1}G^k(\chi',\psi').$$
\end{lemma}

\begin{lemma}[\cite{LN}]
Let $\mathbb{F}_q$ be the finite extension of $\mathbb{F}_p$. Then a multiplicative character $\chi$ of $\mathbb{F}_q$ can be lifted by a multiplicative character $\chi'$ of $\mathbb{F}_p$ if and only if $\chi^{p-1}$ is trivial.
\end{lemma}

Let $\chi_1,\chi_2,\cdots,\chi_s$ be nontrivial multiplicative characters of $\mathbb{F}_q$. The Jacobi sum in $\mathbb{F}_q$ is defined by
$$
J(\chi_1,\chi_2,\cdots,\chi_s)=\sum_{(x_1,x_2,\cdots,x_s)\in \mathbb{F}^s_q \atop x_1+x_2+\cdots+x_s=1}\chi_1(x_1)\chi_2(x_2)\cdots\chi_s(x_s).
$$
The following gives a relation between Gauss sun and Jacobi sum.

\begin{lemma}[\cite{LN}]
Let $\chi_1,\chi_2,\cdots,\chi_s$ be nontrivial multiplicative characters of $\mathbb{F}_q$ with the product $\chi_1\chi_2\cdots\chi_s$ is nontrivial. Let $\psi$ be a
nontrivial additive character of $\mathbb{F}_q$. Then
$$
J(\chi_1,\chi_2,\cdots,\chi_s)=\frac{G(\chi_1,\psi)\cdots G(\chi_s,\psi)}{G(\chi_1\cdots \chi_s,\psi)}.
$$
\end{lemma}

\begin{lemma}[\cite{M2}]
Let $\mathbb{F}_q$ be the finite field of $q=p^k$ elements with the prime $p\equiv 1(\bmod 3)$, and $z$ is non-cubic in $\mathbb{F}^{\ast}_q$,
Then $S(1),S(z)$ and $S(z^2)$ are the roots of the cubic equation
$$x^3-3qx-qc=0,$$
where
$c$ is uniquely determined by
$$4p=c^2+27d^2,\ \ c\equiv 1(\mathrm{mod}3), \ \ (p,c)=1.$$
\end{lemma}

\begin{lemma}[Theorem 3.1.3 of \cite{BEW}]
Let $p\equiv 1(\bmod 3)$ and $\chi'$ be a multiplicative character of order 3 over $\mathbb{F}_p$. Then
$$J(\chi',\chi')=\frac{c_0+3\sqrt{3}d_0\mathrm{i}}{2},$$
where $c_0$ and $d_0$ are uniquely determined by
$$4p=c_0^2+27d_0^2,\ \ c_0\equiv 1(\mathrm{mod}3),\ \ 9d_0\equiv c_0(2g^{\frac{p-1}{3}}+1)(\bmod p)$$
with $g$ being the generator of the multiplicative group $\mathbb{F}^{\ast}_p$ of non-zero residues $(\bmod p)$ such that $\chi'(g)=\frac{-1+\sqrt{3}\mathrm{i}}{2}$.
\end{lemma}

In the rest of this paper, we let $\chi$ be a multiplicative character of order 3 of $\mathbb{F}_q$ and $\psi$ be the canonical additive character which is defined by
$$
\psi(x)=e^{2\pi\mathrm{i}\mathrm{Tr}_{\mathbb{F}_q/\mathbb{F}_p}(x)/p}.
$$
We denote $\overline{\chi}$ the conjugate character of $\chi$. For convenience, we let $G(\chi):=G(\chi,\psi)$. By Lemma 2.2, we have
$G(\chi)G(\overline{\chi})=\chi(-1)q=q$ and $|G(\chi)|=|G(\overline{\chi})|=\sqrt{q}$. We have the following three results for the Gauss sums of order 3.

\begin{lemma}
Let $\mathbb{F}_q$ be the finite field of $q=p^k$ elements with the prime $p\equiv 1(\bmod 3)$. If $z$ is non-cubic in $\mathbb{F}^{\ast}_q$,
then there is a unique multiplicative character $\chi$ of order 3 over $\mathbb{F}_q$ such that
$$\chi(z)=\omega,\ \ J(\chi,\chi)=\frac{c+3\sqrt{3}d\mathrm{i}}{2},\ \ G^3(\chi)=q\cdot\frac{c+3\sqrt{3}d\mathrm{i}}{2},$$
where $\omega=\frac{-1+\sqrt{3}\mathrm{i}}{2}$, $c$ and $d$ are uniquely determined by
\begin{align*}
4q=c^2+27d^2, c\equiv 1(\bmod 3), (c,p)=1, 9d\equiv c(2z^{\frac{q-1}{3}}+1) (\bmod p).
\end{align*}

\end{lemma}

\begin{proof}
Let $g'$ be a generator of the multiplicative group $\mathbb{F}^{\ast}_q$. Note that $z$ is non-cubic. So we have $\mathrm{ind}_{g'}z\equiv \pm 1 (\bmod 3)$.
If $\mathrm{ind}_{g'}z\equiv 1 (\bmod 3)$, we take $g=g'$; If $\mathrm{ind}_{g'}z\equiv -1 (\bmod 3)$, we take $g=(g')^{-1}$.
Hence $g$ also a generator of the group $\mathbb{F}^{\ast}_q$ and $\mathrm{ind}_gz\equiv 1 (\bmod 3)$. Thus we have
\begin{align}\label{z}
z^{\frac{q-1}{3}}= \left(g^{\mathrm{ind}_gz}\right)^{\frac{q-1}{3}}= g^{\frac{q-1}{3}\mathrm{ind}_gz}= g^{\frac{q-1}{3}}.
\end{align}
We take the multiplicative character $\chi(\cdot)=e\left(\frac{\mathrm{ind}_g(\cdot)}{3}\right)$. Obviously, we have
\begin{align*}
\chi(z)=e\left(\frac{\mathrm{ind}_gz}{3}\right)=\chi(g)=e\left(\frac{1}{3}\right)=\omega.
\end{align*}

Since $p\equiv 1(\bmod 3)$, then $\chi^{p-1}$ is trivial. By Lemma 2.4,
the cubic multiplicative character $\chi$ can be lifted by a cubic multiplicative character $\chi'$ of $\mathbb{F}^{\ast}_p$.
It is easy to see that $\mathrm{N}_{\mathbb{F}_q/\mathbb{F}_p}(g)=g^{\frac{q-1}{p-1}}$ is a generator of $\mathbb{F}^{\ast}_p$ and
$$\chi(g)=\chi'(\mathrm{N}_{\mathbb{F}_q/\mathbb{F}_p}(g))=\omega.$$

By Lemma 2.7, we have
$$J(\chi',\chi')=\frac{c_0+3\sqrt{3}d_0\mathrm{i}}{2},$$
where $c_0$ and $d_0$ are uniquely determined by
$$4p=c_0^2+27d_0^2,\ \ c_0\equiv 1(\mathrm{mod}3),\ \ 9d_0\equiv c_0(2(N_{\mathbb{F}_q/\mathbb{F}_p}(g))^{\frac{p-1}{3}}+1)(\mathrm{mod}p).$$
By the Davenport-Hasse Theorem (Lemma 2.3) and Lemma 2.5, we have
\begin{align}\label{J}
J(\chi,\chi)&=(-1)^{k-1}J^k(\chi',\chi')\nonumber\\
&=(-1)^{k-1}\left(\frac{c_0+3\sqrt{3}d_0\mathrm{i}}{2}\right)^k:=\frac{c+3\sqrt{3}d\mathrm{i}}{2}.
\end{align}
So we have $4q=4p^k=c^2+27d^2$ and
\begin{align*}
c&=2\cdot(-1)^{k-1}\mathrm{Re}\left(\frac{c_0+3\sqrt{3}d_0\mathrm{i}}{2}\right)^k=2\cdot(-1)^{k-1}\mathrm{Re}\left(\frac{c_0+3d_0}{2}+3d_0\omega\right)^k\\
&\equiv (-1)^k\left(\frac{c_0+3d_0}{2}\right)^k\equiv c_0^k \equiv 1 (\bmod 3).
\end{align*}

Let $K=\mathbb{Q}(\omega)$. Note that $p\equiv 1(\bmod 3)$. By the prime ideal decomposition of cubic cyclotomic field $K=\mathbb{Q}(\omega)$, we have
\begin{align}\label{p}
pO_K=\left(\frac{c_0+3\sqrt{3}d_0\mathrm{i}}{2}\right)O_K\cdot\left(\frac{c_0-3\sqrt{3}d_0\mathrm{i}}{2}\right)O_K:=P_1P_2.
\end{align}
Thus in $K=\mathbb{Q}(\omega)$, we have the unique decomposition
$$q=\left(\frac{c+3\sqrt{3}d\mathrm{i}}{2}\right)\cdot\left(\frac{c-3\sqrt{3}d\mathrm{i}}{2}\right)
=\left(\frac{c_0+3\sqrt{3}d_0\mathrm{i}}{2}\right)^k\cdot\left(\frac{c_0-3\sqrt{3}d_0\mathrm{i}}{2}\right)^k.$$
Then $c$ is uniquely determined by $4q=c^2+27d^2, c\equiv 1(\bmod 3), (c,p)=1$.

Now we begin to determine the sign of $d$.
Note that $O_K/P_j$ is isomorphic to $\mathbb{F}_p$ for $j=1,2$ and $\mathrm{N}_{\mathbb{F}_q/\mathbb{F}_p}(g)+P_j$ is a generator of $(O_K/P_j)^{\ast}$.
$\left(N_{\mathbb{F}_q/\mathbb{F}_p}(g)\right)^{\frac{p-1}{3}}+P_j$ is a cubic root of unity in $O_K/P_j$.
Then there is one of the prime ideals $P_1$ and $P_2$ (rewrite it as $P$), satisfying
$$\left(N_{\mathbb{F}_q/\mathbb{F}_p}(g)\right)^{\frac{p-1}{3}}\equiv \omega (\bmod P).$$
Thus we have
\begin{align}\label{P}
g^{\frac{q-1}{3}}\equiv \omega (\bmod P).
\end{align}
Define the multiplicative character $\chi_P$ on $(O_K/P)^{\ast}$ by
$$\chi_P(N_{\mathbb{F}_q/\mathbb{F}_p}(g)+P)=\omega.$$
Thus we view $\chi'$ as the character $\chi_P$ on the finite field $O_K/P$ by identifying the generator
$$\chi'(N_{\mathbb{F}_q/\mathbb{F}_p}(g))=\chi_P(N_{\mathbb{F}_q/\mathbb{F}_p}(g)+P)=\omega.$$
Then we have $J(\chi',\chi')=J(\chi_P,\chi_P)$. By Theorem 2.1.14 of \cite{BEW}, we have $J(\chi_P,\chi_P)\equiv 0 (\bmod P)$.
Thus we have
$$J(\chi',\chi')\equiv 0 (\bmod P).$$
So by (\ref{J}), we have
$$J(\chi,\chi)=\frac{c+3\sqrt{3}d\mathrm{i}}{2}=\frac{c+3d(2\omega+1)}{2}\equiv 0 (\bmod P)$$
Then $3d(2\omega+1)\equiv -c(\bmod P)$. Multiplying $-(2\omega+1)$, by (\ref{P}), we have
$$9d\equiv-3d(2\omega+1)^2\equiv c(2\omega+1)\equiv c(2g^{\frac{q-1}{3}}+1) (\bmod P).$$
Hence by (\ref{z}) and (\ref{p}), we have
$$9d\equiv c(2g^{\frac{q-1}{3}}+1)\equiv c(2z^{\frac{q-1}{3}}+1) (\bmod p).$$
Since $\chi$ is a multiplicative character of order 3, by Lemma 2.5, we have
$$G^3(\chi)=J(\chi,\chi)G(\chi^2)G(\chi)=J(\chi,\chi)G(\overline{\chi})G(\chi)=qJ(\chi,\chi).$$
This completes the proof of Lemma 2.8.
\end{proof}

\begin{lemma}
Let $\chi$ be a multiplicative character of order 3 of $\mathbb{F}_q$. Then for any $a\in \mathbb{F}^{\ast}_q$, we have
\begin{align}\label{S(a)}
S(a)=\overline{\chi}(a)G(\chi)+\chi(a)G(\overline{\chi}).
\end{align}
\end{lemma}

\begin{proof}
Note that $\chi$ be the multiplicative character of order 3. Then we have
$$1+\chi(k)+\overline{\chi}(k)=\left\{
                        \begin{array}{ll}
                          3, & \hbox{if $k$ is cubic;} \\
                          0, & \hbox{if $k$ is non-cubic.}
                        \end{array}
                      \right.
$$
Thus for any $a\in \mathbb{F}^{\ast}_q$, we have
\begin{align*}
S(a)=\sum_{k\in \mathbb{F}^{\ast}_q}\psi(ak^3)
&=1+\sum_{k\in \mathbb{F}^{\ast}_q}(1+\chi(k)+\overline{\chi}(k))\psi(ak)\\
&=1+\sum_{k\in \mathbb{F}^{\ast}_q}\psi(ak)+\sum_{k\in \mathbb{F}^{\ast}_q}\chi(k)\psi(ak)+\sum_{k\in \mathbb{F}^{\ast}_q}\overline{\chi}(k)\psi(ak)\\
&=\overline{\chi}(a)\sum_{k\in \mathbb{F}^{\ast}_q}\chi(ak)\psi(ak)+\chi(a)\sum_{k\in \mathbb{F}^{\ast}_q}\overline{\chi}(ak)\psi(ak)\\
&=\overline{\chi}(a)G(\chi)+\chi(a)G(\overline{\chi}).
\end{align*}
\end{proof}

\begin{lemma}
Let $\mathbb{F}_q$ be the finite field of $q=p^k$ elements with the prime $p\equiv 1(\bmod 3)$. If $z$ is non-cubic in $\mathbb{F}^{\ast}_q$,
then
$$
S(1)^2S(z)+S(z)^2S(z^2)+S(z^2)^2S(1)=\frac{3}{2}q(9d-c),
$$
where $c$ and $d$ are uniquely determined by (\ref{cd}).
\end{lemma}
\begin{proof}
Since $p\equiv 1(\bmod 3)$, the non-zero cubic elements form a multiplicative subgroup $H$ of order $\frac{1}{3}(q-1)$ and index 3 which partitions $\mathbb{F}^{\ast}_q$ into
three cosets $H, zH$ and $z^2H$. Then for any $a\in z^jH$, we have $S(a)=S(z^j)$ and $S(az)=S(z^{j+1})$. Thus we have
\begin{align}\label{2.10-1}
\sum_{a\in\mathbb{F}^{\ast}_q}S(a)^2S(az)&=\sum_{a\in H}S(a)^2S(az)+\sum_{a\in zH}S(a)^2S(az)+\sum_{a\in zH}S(a)^2S(az)\nonumber\\
&=\frac{1}{3}(q-1)\left(S(1)^2S(z)+S(z)^2S(z^2)+S(z^2)^2S(1)\right).
\end{align}
On the other hand, by Lemma 2.8,
there is a unique multiplicative character $\chi$ of order 3 over $\mathbb{F}_q$ such that
$$\chi(z)=\frac{-1+\sqrt{3}\mathrm{i}}{2},\ \ G^3(\chi)=q\cdot\frac{c+3\sqrt{3}d\mathrm{i}}{2},$$
where $c$ and $d$ are uniquely determined by (\ref{cd}).
By Lemmas 2.1 and 2.9, we have
\begin{align}\label{2.10-2}
&\sum_{a\in\mathbb{F}^{\ast}_q}S(a)^2S(az)\nonumber\\
&=\sum_{a\in\mathbb{F}^{\ast}_q}(\overline{\chi}(a)G(\chi)+\chi(a)G(\overline{\chi}))^2(\overline{\chi}(az)G(\chi)+\chi(az)G(\overline{\chi}))\nonumber\\
&=(q-1)\left(\overline{\chi}(z)G^3(\chi)+\chi(z)G^3(\overline{\chi})\right)\nonumber\\
&=q(q-1)\left(\frac{-1-\sqrt{3}\mathrm{i}}{2}\cdot\frac{c+3\sqrt{3}d\mathrm{i}}{2}+\frac{-1+\sqrt{3}\mathrm{i}}{2}\cdot\frac{c-3\sqrt{3}d\mathrm{i}}{2}\right)\nonumber\\
&=\frac{1}{2}q(q-1)(9d-c).
\end{align}
Then Lemma 2.10 immediately follows from (\ref{2.10-1}) and (\ref{2.10-2}).

\end{proof}

\section{Proofs of Theorems 1.4 and 1.5}
In this section, we prove Theorem 1.4 and 1.5. First, we begin with the proof of Theorem 1.5.

{\it Proof of Theorem 1.5}.
By Remark 1.6, we only need to consider the case $q\equiv 1(\bmod 3)$ with $p\equiv 1(\bmod 3)$.
By Lemma 2.1, we have
\begin{align*}
B_n(z)&=\frac{1}{q}\sum_{a\in \mathbb{F}_q}\sum_{(x_1,x_2,\cdots,x_{n+1})\in \mathbb{F}^{s+1}_q}\psi\left(a(x_1^3+\cdots+x_n^3+zx_{n+1}^3)\right)\\
&=q^n+\frac{1}{q}\sum_{a\in \mathbb{F}^{\ast}_q}(S(a))^nS(az).
\end{align*}
Then
\begin{align*}
\sum_{n=0}^{\infty}B_n(z)x^n
&=\sum_{n=0}^{\infty}q^nx^n+\frac{1}{q}\sum_{a\in \mathbb{F}^{\ast}_q}S(az)\sum_{n=0}^{\infty}(S(a))^nx^n\\
&=\frac{1}{1-qx}+\frac{1}{q}\sum_{a\in \mathbb{F}^{\ast}_q}\frac{S(az)}{1-S(a)x}.
\end{align*}
Since $p\equiv 1(\bmod 3)$, the non-zero cubic elements form a multiplicative subgroup $H$ of order $\frac{1}{3}(q-1)$ and index 3. Then by the proof of Lemma 2.10, we have
\begin{align*}
&\sum_{n=0}^{\infty}B_n(z)x^n\\
&=\frac{1}{1-qx}+\frac{1}{q}\left(\sum_{a\in H}\frac{S(az)}{1-S(a)x}+\sum_{a\in zH}\frac{S(az)}{1-S(a)x}+\sum_{a\in z^2H}\frac{S(az)}{1-S(a)x}\right)\\
&=\frac{1}{1-qx}+\frac{q-1}{3q}\left(\frac{S(z)}{1-S(1)x}+\frac{S(z^2)}{1-S(z)x}+\frac{S(1)}{1-S(z^2)x}\right)\\
&=\frac{1}{1-qx}+\frac{q-1}{3q}\cdot\frac{\alpha-(\alpha^2-\beta)x+\gamma x^2}{1-\alpha x+\beta x^2-\delta x^3},
\end{align*}
Where $\alpha=S(1)+S(z)+S(z^2)$, $\beta=S(1)S(z)+S(z)S(z^2)+S(z^2)S(1)$, $\gamma=S(1)^2S(z)+S(z)^2S(z^2)+S(z^2)^2S(1)$ and $\delta=S(1)S(z)S(z^2)$.
By Lemmas 2.6 and 2.10, we have
$$\alpha=0,\beta=-3q,\gamma=\frac{3}{2}q(9d-c),\delta=qc.$$
Thus we have
\begin{align*}
\sum_{n=0}^{\infty}B_n(z)x^n&=\frac{1}{1-qx}+\frac{q-1}{3q}\cdot\frac{-3qx+\frac{3}{2}q(9d-c)x^2}{1-3qx^2-qcx^3}\\
&=\frac{1}{1-qx}-\frac{(q-1)x+\frac{1}{2}(q-1)(c-9d)x^2}{1-3qx^2-qcx^3}.
\end{align*}
This completes the proof of the Theorem 1.5.

{\it Proof of Theorem 1.4}.
By the proof of Theorem 1.3 in \cite{HZ}, it is easy to see that
$$B_n(z)=A_n(0)+(q-1)A_n(z).$$
Thus we have
$$A_n(z)=\frac{1}{(q-1)}(B_n(z)-A_n(0)).$$
If $z$ is non-cubic, then by Theorems 1.1 and 1.5, we have
\begin{align*}
\sum_{n=1}^{\infty}A_n(z)x^n
&=\frac{1}{(q-1)}\left(\sum_{n=1}^{\infty}B_n(z)x^n-\sum_{n=1}^{\infty}A_n(0)x^n\right)\\
&=\frac{1}{(q-1)}\left(\frac{1}{1-qx}-\frac{(q-1)x+\frac{1}{2}(q-1)(c-9d)x^2}{1-3qx^2-qcx^3}-B_0(z)\right)\\
&\ \ \ \ -\frac{1}{(q-1)}\left(\frac{x}{1-qx}+\frac{(q-1)(2+cx)x^2}{1-3qx^2-qcx^3}\right)\\
&=\frac{x}{1-qx}-\frac{x+\frac{1}{2}(4+c-9d)x^2+cx^3}{1-3qx^2-qcx^3}.
\end{align*}

If $z$ is cubic, we have $B_n(z)=A_{n+1}(0)$. By Theorem 1.1, we have
\begin{align*}
\sum_{n=1}^{\infty}A_n(z)x^n
&=\frac{1}{(q-1)}\left(\sum_{n=1}^{\infty}A_{n+1}(0)x^n-\sum_{n=1}^{\infty}A_n(0)x^n\right)\\
&=\frac{1}{(q-1)}\left(\frac{1}{x}\sum_{n=1}^{\infty}A_n(0)x^n-A_1(0)-\sum_{n=1}^{\infty}A_n(0)x^n\right)\\
&=\frac{1}{(q-1)}\left(\frac{1-x}{x}\sum_{n=1}^{\infty}A_n(0)x^n-1\right)\\
&=\frac{x}{1-qx}+\frac{2x+(c-2)x^2-cx^3}{1-3qx^2-qcx^3}.
\end{align*}
This completes the proof of the Theorem 1.4.

\section{Proofs of Theorems 1.9 and 1.10 and an example}
In this section, we prove Theorem 1.9 and 1.10. First, we begin with the proof of Theorem 1.9.

{\it Proof of Theorem 1.9}.

By Lemma 2.1, we have
\begin{align*}
M_k(a_1,a_2,a_3)&=\frac{1}{q}\sum_{m\in \mathbb{F}_q}\sum_{(x_1,x_2,x_3)\in \mathbb{F}^3_q}\psi\left(m(a_1x_1^3+a_2x_2^3+a_3x_3^3)\right)\\
&=q^2+\frac{1}{q}\sum_{m\in \mathbb{F}^{\ast}_q}S(a_1m)S(a_2m)S(a_3m).
\end{align*}

Then by Lemma 2.9, for any multiplicative character $\chi$ of order 3, we have
\begin{align*}
M_k(a_1,a_2,a_3)&=q^2+\frac{1}{q}\sum_{m\in \mathbb{F}^{\ast}_q}\left[\prod_{j=1}^3\left(\overline{\chi}(ma_j)G(\chi)+\chi(ma_j)G(\overline{\chi})\right)\right]\\
&=q^2+\frac{1}{q}\sum_{m\in \mathbb{F}^{\ast}_q}\left[\overline{\chi}(a_1a_2a_3)G^3(\chi)+\chi(a_1a_2a_3)G^3(\overline{\chi})\right]\\
&+G(\chi)(\chi(a^{-1}_1a^{-1}_2a_3)+\chi(a^{-1}_1a_2a^{-1}_3)+\chi(a_1a^{-1}_2a^{-1}_3))\sum_{m\in \mathbb{F}^{\ast}_q}\overline{\chi}(m)\\
&+G(\overline{\chi})(\chi(a^{-1}_1a_2a_3)+\chi(a_1a^{-1}_2a_3)+\chi(a_1a_2a^{-1}_3))\sum_{m\in \mathbb{F}^{\ast}_q}\chi(m)\\
&=q^2+\frac{q-1}{q}\left[\overline{\chi}(a_1a_2a_3)G^3(\chi)+\chi(a_1a_2a_3)G^3(\overline{\chi})\right].
\end{align*}
If $a_1a_2a_3$ is cubic, thus we have $\chi(a_1a_2a_3)=\overline{\chi}(a_1a_2a_3)=1$. then by Lemma 2.8, we have
\begin{align*}
M_k(a_1,a_2,a_3)&=q^2+\frac{q-1}{q}(G^3(\chi)+G^3(\overline{\chi}))\\
&=q^2+(q-1)\left[\frac{c+3\sqrt{3}d\mathrm{i}}{2}+\frac{c-3\sqrt{3}d\mathrm{i}}{2}\right]\\
&=q^2+c(q-1).
\end{align*}
If $a_1a_2a_3$ is non-cubic, then by Lemma 2.8, we can take multiplicative character $\chi$ of order 3 satisfying
$$\chi(a_1a_2a_3)=\frac{-1+\sqrt{3}\mathrm{i}}{2},\ \ G^3(\chi)=q\cdot\frac{c+3\sqrt{3}d\mathrm{i}}{2},$$
where $c$ and $d$ are uniquely determined by (\ref{cd for a}).
Thus we have
\begin{align*}
M_k(a_1,a_2,a_3)&=q^2+(q-1)\left(\frac{-1-\sqrt{3}\mathrm{i}}{2}\cdot\frac{c+3\sqrt{3}d\mathrm{i}}{2}+\frac{-1+\sqrt{3}\mathrm{i}}{2}\cdot\frac{c-3\sqrt{3}d\mathrm{i}}{2}\right)\\
&=q^2+\frac{1}{2}(q-1)(9d-c).
\end{align*}

This completes the proof of the Theorem 1.9.

{\it Proof of Theorem 1.10}. We have
\begin{align*}
M_k(a_1,a_2,a_3)&=\sum_{(x_1,x_2,x_3)\in \mathbb{F}^3_q \atop a_1x_1^3+a_2x_2^3+a_3x_3^3=0}1\\
&=\sum_{(x_1,x_2)\in \mathbb{F}^2_q, x_3\in \mathbb{F}^{\ast}_q\atop a_1x_1^3+a_2x_2^3+a_3x_3^3=0}1
+\sum_{(x_1,x_2)\in \mathbb{F}^2_q \atop a_1x_1^3+a_2x_2^3=0}1\\
&=\sum_{(x_1,x_2)\in \mathbb{F}^2_q, x_3\in \mathbb{F}^{\ast}_q \atop a_1(-x_1x^{-1}_3)^3+a_2(-x_2x^{-1}_3)^3=a_3}1+
\sum_{x_1\in \mathbb{F}^{\ast}_q,x_2\in \mathbb{F}_q \atop a_1x_1^3+a_2x_2^3=0}1+1\\
&=(q-1)\sum_{(x_1,x_2)\in \mathbb{F}^2_q \atop a_1x_1^3+a_2x_2^3=a_3}1+(q-1)\sum_{x\in \mathbb{F}_q \atop x^3=-a_1a_2^{-1}}1+1\\
&=(q-1)N_k(a_1,a_2,a_3)+1+(q-1)\sum_{x\in \mathbb{F}_q \atop x^3=-a_1a_2^{-1}}1.
\end{align*}

If $a_1a_2^{-1}$ is cubic, the number of solutions of the equation $x^3=-a_1a_2^{-1}$ is exactly 3. Thus we have
$$M_k(a_1,a_2,a_3)=(q-1)N_k(a_1,a_2,a_3)+1+3(q-1)=(q-1)N_k(a_1,a_2,a_3)+3q-2.$$

If $a_1a_2^{-1}$ is non-cubic, the equation $x^3=-a_1a_2^{-1}$  has no solution. Thus we have
$$M_k(a_1,a_2,a_3)=(p-1)N_k(a_1,a_2,a_3)+1.$$
Hence Theorem 1.10 immediately follows from Theorem 1.9.

\begin{example} We take $\mathbb{F}_{7^2}:=\mathbb{F}_7[u]/(u^2+1)$. One can check that $u+1$ is non-cubic in $\mathbb{F}_7[u]/(u^2+1)$ and $(u+1)^{\frac{7^2-1}{3}}=4$.
If the integers $c$ and $d$ satisfying that $4\cdot 7^2=c^2+27d^2, c\equiv 1 (\bmod 3), (c,p)=1, 9d\equiv c(2(u+1)^{\frac{q-1}{3}}+1)(\bmod p)$, then $c=13, d=-1$. Thus
we have
$$N_2(1,1,u+1)=7^2-2+\frac{1}{2}(-9-13)=36$$
and
$$M_2(1,1,u+1)=49^2+\frac{1}{2}(-9-13)=1873.$$
We list the solutions of equation $x_1^3+x_2^3=u+1$ over $\mathbb{F}_7[u]/(u^2+1)$ as belove:
\begin{align*}
&(1,3u);(1,5u);(1,6u);(2,3u);(2,5u);(2,6u);(4,3u);(4,5u);(4,6u);\\
&(u+4,3u+6);(u+4,5u+3);(u+4,6u+5);(2u+1,3u+6);(2u+1,5u+3);\\
&(2u+1,6u+5);(4u+2,3u+6);(4u+2,5u+3);(4u+2,6u+5),
\end{align*}
and one can get the remaining 18 solutions by exchanging coordinates.

\end{example}

\section*{Acknowledgments}
The authors are partially supported by the National Natural Science Foundation of China (Grant
No. 11871193, 12071132) and the Natural Science Foundation of Henan Province (No. 202300410031).

\end{document}